\documentclass[12pt]{amsart}

\usepackage{amsmath}
\usepackage{amscd}
\usepackage{amssymb}
\usepackage{amsfonts}
\usepackage{amsthm}
\usepackage{enumerate}
\usepackage{hyperref}
\usepackage{color}
\usepackage{graphicx}
\theoremstyle{plain}
\newtheorem{thm}{Theorem}[section]

\newtheorem{prop}[thm]{Proposition}
\newtheorem{lem}[thm]{Lemma}

\theoremstyle{definition}

\newtheorem{dfns-rems}[thm]{Definitions and Remarks}
\newtheorem{notas-rems}[thm]{Notations and Remarks}
\newtheorem{exmps-rems}[thm]{Examples and Remarks}

\begin{document}

% ------------------------------------------------------------------------

\title[Symbolic powers of edge ideals of unicyclic graphs]{Regularity of symbolic powers of edge ideals of unicyclic graphs}

% ------------------------------------------------------------------------

\author[S. A. Seyed Fakhari]{S. A. Seyed Fakhari}

\address{S. A. Seyed Fakhari, School of Mathematics, Statistics and Computer Science,
College of Science, University of Tehran, Tehran, Iran, and Institute of Mathematics, Vietnam Academy of Science and Technology, 18 Hoang Quoc Viet, Hanoi, Vietnam.}

\email{aminfakhari@ut.ac.ir}

% ------------------------------------------------------------------------

\begin{abstract}
Let $G$ be a unicyclic graph with edge ideal $I(G)$. For any integer $s\geq 1$, we denote the $s$-th symbolic power of $I(G)$ by $I(G)^{(s)}$. It is shown that ${\rm reg}(I(G)^{(s)})={\rm reg}(I(G)^s)$, for every $s\geq 1$.
\end{abstract}

% ------------------------------------------------------------------------


% ------------------------------------------------------------------------


% ------------------------------------------------------------------------

\thanks{This research is partially funded by the Simons Foundation Grant Targeted for Institute of Mathematics, Vietnam Academy of Science and Technology.}

% ------------------------------------------------------------------------

\maketitle

%%%%%%%%%%%%%%%%%%%%%%%%%%%%%%%%%%%%%%%%%%%%%%%%%%%%%%%%%%%%%%%%%%%%%%%%%%

\section{Introduction} \label{sec1}

Let $\mathbb{K}$ be a field and $S = \mathbb{K}[x_1,\ldots,x_n]$  be the
polynomial ring in $n$ variables over $\mathbb{K}$. Suppose that $M$ is a graded $S$-module with minimal free resolution
$$0  \longrightarrow \cdots \longrightarrow  \bigoplus_{j}S(-j)^{\beta_{1,j}(M)} \longrightarrow \bigoplus_{j}S(-j)^{\beta_{0,j}(M)}   \longrightarrow  M \longrightarrow 0.$$
The Castelnuovo--Mumford regularity (or simply, regularity) of $M$, denote by ${\rm reg}(M)$, is defined as follows:
$${\rm reg}(M)=\max\{j-i|\ \beta_{i,j}(M)\neq0\}.$$
The regularity of $M$ is an important invariant in commutative algebra and algebraic geometry.

There is a natural correspondence between quadratic squarefree monomial ideals of $S$ and finite simple graphs with $n$ vertices. To every simple graph $G$ with vertex set $V(G)=\big\{x_1, \ldots, x_n\big\}$ and edge set $E(G)$, we associate its {\it edge ideal} $I=I(G)$ defined by
$$I(G)=\big(x_ix_j: x_ix_j\in E(G)\big)\subseteq S$$(by abusing the notation, we identify the edges of $G$ with quadratic monomials). Computing and finding bounds for the regularity of edge ideals and their powers have been studied by a number of researchers (see for example \cite{ab},  \cite{abs}, \cite{b}, \cite{bbh1}, \cite{bht}, \cite{dhs}, \cite{ha}, \cite{jns}, \cite{k}, \cite{khm}, \cite{msy}, \cite{s7}, \cite{sy} and \cite{wo}).

Katzman \cite{k}, proved that for any graph $G$,
\[
\begin{array}{rl}
{\rm reg}(I(G))\geq \nu(G)+1,
\end{array} \tag{$\dagger$} \label{dag}
\]
where $\nu(G)$ denotes the induced matching number of $G$. Beyarslan, H${\rm \grave{a}}$ and Trung \cite{bht}, generalized Katzman's inequality by showing that$${\rm reg}(I(G)^s)\geq 2s+\nu(G)-1,$$for every integer $s\geq 1$. In the same paper, the authors  proved the equality for every $s\geq 2$, if $G$ is a cycle (see \cite[Theorem 5.2]{bht}).

This paper is motivated by a conjecture of N. C. Minh, who predicted that for any graph $G$ and every integer $s\geq 1$, the equality$${\rm reg}(I(G)^{(s)})={\rm reg}(I(G)^s)$$holds, where $I(G)^{(s)}$ denotes the $s$-th symbolic power of $I(G)$ (see \cite{ghos}). If $G$ is a bipartite graph, by \cite[Theorem 5.9]{svv}, we have $I(G)^{(s)}=I(G)^s$, for any $s\geq 1$. Thus, the conjecture of Minh is trivially true in this case. If $G$ is not a bipartite graph, then it contains an odd cycle. Therefore, the first case of study to verify Minh's conjecture is the class of odd cycle graphs and this has been already done by Gu, H${\rm \grave{a}}$, O'Rourke and Skelton \cite{ghos}. In fact, they proved in \cite[Theorem 5.3]{ghos} that for any odd cycle graph $G$ and every integer $s\geq 1$, we have ${\rm reg}(I(G)^{(s)})={\rm reg}(I(G)^s)$. The next case is the class of unicyclic graphs. We know from \cite[Corollary 4.12]{bc} that for every unicyclic graph $G$, the regularity of $I(G)$ is either $\nu(G)+1$ or $\nu(G)+2$. Alilooee, Beyarslan and Selvaraja \cite{abs} characterized the unicyclic graphs with regularity $\nu(G)+1$ and $\nu(G)+2$. They also proved that for any unicyclic graph (which is not a cycle) and any integer $s\geq 1$, we have$${\rm reg}(I(G)^s)=2s+{\rm reg}(I(G))-2.$$ Recently, Jayanthan and Kumar \cite{jk}, proved Minh's conjecture for a subclass of unicyclic graphs. As the main result of this paper, in Theorem \ref{main}, we prove the conjectured equality of Minh for any arbitrary unicyclic graph. More precisely, we show that for every unicyclic graph $G$ and every integer $s\geq 1$,$${\rm reg}(I(G)^{(s)})={\rm reg}(I(G)^s).$$

%%%%%%%%%%%%%%%%%%%%%%%%%%%%%%%%%%%%%%%%%%%%%%%%%%%%%%%%%%%%%%%%%%%%%%%%%%

\section{Preliminaries} \label{sec2}

In this section, we provide the definitions and basic facts which will be used in the next section.

Let $G$ be a simple graph with vertex set $V(G)=\big\{x_1, \ldots,
x_n\big\}$ and edge set $E(G)$. For a vertex $x_i$, the {\it neighbor set} of $x_i$ is $N_G(x_i)=\{x_j\mid x_ix_j\in E(G)\}$ and we set $N_G[x_i]=N_G(x_i)\cup \{x_i\}$. The cardinality of $N_G(x_i)$ is called the {\it degree} of $x_i$. A vertex of degree one is a {\it leaf} and the edge incident to a leaf is a {\it pendant edge}. A {\it forest} is a graph with no cycle. The graph $G$ is called {\it unicyclic} if it has exactly one cycle as a subgraph. The {\it length} of a path of or a cycle is the number of its edges. For any pair of vertices $x,y\in V(G)$, the {\it distance} of $x$ and $y$, denoted by $d_G(x,y)$ is the length of the shortest path between $x$ and $y$ in $G$. For a subset $W\subseteq V(G)$ and a vertex $x\in V(G)$ the {\it distance} of $x$ and $W$ is defined as$$d_G(x,W)={\rm min}\{d_G(x,y)\mid y\in W\}.$$For every subset $U\subset V(G)$, the graph $G\setminus U$ has vertex set $V(G\setminus U)=V(G)\setminus U$ and edge set $E(G\setminus U)=\{e\in E(G)\mid e\cap U=\emptyset\}$. If $U=\{x\}$ is a singleton, then we write $G-x$ instead of $G\setminus \{x\}$. A subgraph $H$ of $G$ is called {\it induced} provided that two vertices of $H$ are adjacent if and only if they are adjacent in $G$. A subset $A$ of $V(G)$ is a {\it vertex cover} of $G$ if every edge of $G$ is incident to at least one vertex of $A$. A vertex cover $A$ is a {\it minimal vertex cover} if no proper subset of $A$ is a vertex cover of $G$. The set of minimal vertex covers of $G$ will be denoted by $\mathcal{C}(G)$.

For every subset $A$ of $\big\{x_1, \ldots, x_n\big\}$, we denote by $\mathfrak{p}_A$, the monomial prime ideal which is generated by the variables belonging to $A$. It is well-known that for every graph $G$ with edge ideal $I(G)$,$$I(G)=\bigcap_{A\in \mathcal{C}(G)}\mathfrak{p}_A.$$

Let $G$ be a graph. A subset $M\subseteq E(G)$ is a {\it matching} if $e\cap e'=\emptyset$, for every pair of edges $e, e'\in M$. A matching $M$ of $G$ is an {\it induced matching} of $G$ if for every pair of edges $e, e'\in M$, there is no edge $f\in E(G)\setminus M$ with $f\subset e\cup e'$. The cardinality of the largest induced matching of $G$ is called the {\it induced  matching number} of $G$ and is denoted by $\nu(G)$.

Let $I$ be an ideal of $S$ and let ${\rm Min}(I)$ denote the set of minimal primes of $I$. For every integer $s\geq 1$, the $s$-th {\it symbolic power} of $I$,
denoted by $I^{(s)}$, is defined to be$$I^{(s)}=\bigcap_{\frak{p}\in {\rm Min}(I)} {\rm Ker}(R\rightarrow (R/I^s)_{\frak{p}}).$$Assume that $I$ is a squarefree monomial ideal in $S$ and suppose $I$ has the irredundant
primary decomposition $$I=\frak{p}_1\cap\ldots\cap\frak{p}_r,$$ where every
$\frak{p}_i$ is an ideal generated by a subset of the variables of
$S$. It follows from \cite[Proposition 1.4.4]{hh} that for every integer $s\geq 1$, $$I^{(s)}=\frak{p}_1^s\cap\ldots\cap
\frak{p}_r^s.$$

%%%%%%%%%%%%%%%%%%%%%%%%%%%%%%%%%%%%%%%%%%%%%%%%%%%%%%%%%%%%%%%%%%%%%%%%%%

\section{Main results} \label{sec3}

In this section, we prove the main result of this paper, Theorem \ref{main}, which states that for every unicyclic graph $G$ and every integer $s\geq 1$, the regularity of the $s$-th ordinary and symbolic powers of $I(G)$ are equal. The most technical part of the proof is Proposition \ref{sum} and we need a series of lemmata in order to prove this proposition.

\begin{lem} \label{col}
Let $G$ be a graph and $x$ be a leaf of $G$. Assume that $y$ is the unique neighbor of $x$. Then for every integer $s\geq 1$, we have$${\rm reg}(I(G\setminus N_G[y])^s)\leq {\rm reg}(I(G)^s)-1.$$
\end{lem}

\begin{proof}
Let $H$ be the induced subgraph of $G$ on the vertices $V(G\setminus N_G[y])\cup \{x,y\}$. Then $H$ is the disjoint union of $G\setminus N_G[y]$ and the edge $xy$. Therefore,$${\rm reg}(I(G\setminus N_G[y])^s)+1\leq {\rm reg}(I(H)^s)\leq {\rm reg}(I(G)^s),$$where the first inequality follows from \cite[Theorem 1.1]{nv}, and the second inequality follows from \cite[Corollay 4.3]{bht}.
\end{proof}

\begin{lem} \label{del}
Let $G$ be a graph and $xy$ be an edge of $G$. Assume that $G'$ is the graph obtained from $G$ by deleting the edge $xy$. Then for every integer $s\geq 1$,$$I(G)^{(s)}+(xy)=I(G')^{(s)}+(xy).$$
\end{lem}

\begin{proof}
It is clear that$$I(G')^{(s)}+(xy)\subseteq I(G)^{(s)}+(xy).$$Therefore, we only need to prove the reverse inclusion. Assume that $u\in I(G)^{(s)}+(xy)$ is a monomial. The assertion is obvious, if $u$ is divisible by $xy$. Thus, suppose $u$ is not divisible by $xy$. In other words, at least one of the variables $x$ and $y$ does not divide $u$. Let $z\in \{x, y\}$ be a variable with the property that $z\nmid u$. Then$$u\in I(G)^{(s)}\subset I(G)^{(s)}+(z)=I(G\setminus z)^{(s)}+(z)\subseteq I(G')^{(s)}+(z),$$and since $u$ is not divisible by $z$, we deduce that $u\in I(G')^{(s)}$.
\end{proof}

\begin{lem} \label{colon}
Let $G$ be a graph and $x$ be a leaf of $G$. Assume that $y$ is the unique neighbor of $x$. Then for every integer $s\geq 1$,$$(I(G)^{(s)}:xy)=I(G)^{(s-1)}.$$
\end{lem}

\begin{proof}
Let $A$ be a minimal vertex cover of $G$. Then $A$ contains exactly one of $x$ and $y$. In other words, $|A\cap \{x,y\}|=1$. Consequently, $$(\mathfrak{p}_A^s: xy)=\mathfrak{p}_A^{s-1}.$$We know that $I(G)=\bigcap_{A\in \mathcal{C}(G)}\mathfrak{p}_A$. Thus, $$I(G)^{(s)}=\bigcap_{A\in \mathcal{C}(G)}\mathfrak{p}_A^s.$$It follows that$$(I(G)^{(s)}: xy)=\bigcap_{A\in \mathcal{C}(G)}(\mathfrak{p}_A^s: xy)=\bigcap_{A\in \mathcal{C}(G)}\mathfrak{p}_A^{s-1}=I(G)^{(s-1)}.$$
\end{proof}

In the next two lemma, we consider a subclass of unicyclic graphs.

\begin{lem} \label{dist}
Let $G$ be a unicyclic graph and $C$ be its unique cycle. Assume further that for every leaf $x$ of $G$, we have $d_G(x,C)=1$. Then for any integer $s\geq 1$, we have$${\rm reg}(I(G)^{(s)})\leq {\rm reg}(I(G)^s).$$
\end{lem}

\begin{proof}
If $C$ is an even cycle, then $G$ is a bipartite graph and by \cite[Theorem 5.9]{svv}, we have $I(G)^{(s)}=I(G)^s$. Hence, there is nothing to prove in this case. Therefore, suppose $C$ is an odd cycle.

\vspace{0.3cm}
{\bf Claim.} We have$$I(G)^{(s)}\cap \mathfrak{m}^{2s}=I(G)^s,$$ where $\mathfrak{m}=(x_1,\ldots,x_n)$ is the graded maximal ideal of $S$.

{\it Proof of the claim.} The inclusion$$I(G)^s \subseteq I(G)^{(s)}\cap \mathfrak{m}^{2s}$$is trivial. To prove the other inclusion, let $u$ be a monomial in $I(G)^{(s)}\cap \mathfrak{m}^{2s}$. Without lose of generality, assume that $V(C)=\{x_1, \ldots, x_{2m+1}\}$, for some integer $m\geq 1$. Set $v=x_1 \ldots x_{2m+1}$. By \cite[Theorem 3.4]{ghos}, there exists an integer $t$ with $0\leq t\leq \lfloor\frac{s}{m+1}\rfloor$, with $u\in v^tI(G)^{s-t(m+1)}$. Since $u\in \mathfrak{m}^{2s}$, we can write $u=v^tu_1u_2$, where $u_1$ is a minimal monomial generator of $I(G)^{s-t(m+1)}$ and $u_2$ is a monomial with ${\rm deg}(u_2)\geq t$. By assumption, for every vertex $x\in V(G)$, either $x\in V(C)$ or $d_G(x,C)=1$. In both cases $xv\in I(G)^{m+1}$. In particular, $u_2v^t\in I(G)^{t(m+1)}$. This implies that$$u=v^tu_1u_2\in I(G)^{t(m+1)}I(G)^{s-t(m+1)}=I(G)^s,$$and this proves the claim.

\vspace{0.3cm}
We know consider the following exact sequence.
\begin{align*}
0 \longrightarrow \frac{S}{I(G)^s}\longrightarrow \frac{S}{I(G)^{(s)}}\oplus \frac{S}{\mathfrak{m}^{2s}} \longrightarrow \frac{S}{I(G)^{(s)}+\mathfrak{m}^{2s}}
\longrightarrow 0
\end{align*}
Note that ${\rm reg}(S/\mathfrak{m}^{2s})=2s-1$ and ${\rm reg}(S/(I(G)^s+\mathfrak{m}^{2s}))\leq 2s-1$. Thus,
\begin{align*}
& {\rm reg}(S/I(G)^{(s)})={\rm reg}(S/I(G)^{(s)}\oplus S/\mathfrak{m}^{2s})\\ & \leq \max \big\{{\rm reg}(S/I(G)^s),{\rm reg}\big(S/(I(G)^s+\mathfrak{m}^{2s})\big)\big\}\\ & ={\rm reg}(S/I(G)^s).
\end{align*}
\end{proof}

\begin{lem} \label{first1}
Suppose $G$ be a unicyclic graph and $C$ is its unique cycle. Assume further that for every leaf $x$ of $G$, we have $d_G(x,C)=1$. Let $H$ be a subgraph of $G$ with $E(H)\subseteq E(G)\setminus E(C)$. Then for every integer $s\geq 1$, we have$${\rm reg}(I(G)^{(s)}+I(H))\leq {\rm reg}(I(G)^s).$$
\end{lem}

\begin{proof}
We prove the assertion by induction on $s+|E(H)|$. There is nothing to prove for $s=1$, as $I(G)+I(H)=I(G)$. Therefore, assume that $s\geq 2$. For $E(H)=\emptyset$, the assertion follows from Lemma \ref{dist}. Hence, suppose $|E(H)|\geq 1$. Let $e\in E(H)$ be an arbitrary edge. By assumption, $e$ is a pendant edge of $G$. Without loss of generality, suppose $e=xy$, where $x$ is a leaf of $G$. Then $y\in V(C)$. Let $H'$ be the graph obtained from $H$ by deleting the edge $e$. Set $U:=N_{H'}(y)$. Note that every vertex belonging to $U$ is a leaf of $G$. In particular, $U\cap V(C)=\emptyset$. Since $x$ is a leaf of $G$, we deduce that $H'$ has no edge incident to $x$. Hence,$$(I(H'):xy)=(I(H'):y)=I(H'\setminus N_{H'}[y])+({\rm the \ ideal\ generated\ by}\ U).$$Set $G':=G\setminus U$. Then $G'$ is unicyclic graph, which satisfies the assumptions of the lemma. Using Lemma \ref{colon}, we conclude that
\begin{align*}
& {\rm reg}\big((I(G)^{(s)}+I(H')):xy\big)={\rm reg}\big(I(G\setminus U)^{(s-1)}+I(H'\setminus N_{H'}[y])\big)\\ & \leq {\rm reg}(I(G')^{s-1})\leq {\rm reg}(I(G)^{s-1}),
\end{align*}
where the first inequality follows from the induction hypothesis and the second inequality follows from \cite[Corollary 4.3]{bht}. It then follows from \cite[Theorem 5.4]{abs} that$${\rm reg}\big((I(G)^{(s)}+I(H')):xy\big)\leq {\rm reg}(I(G)^{s-1})={\rm reg}(I(G)^s)-2.$$

Using the induction hypothesis, we also have$${\rm reg}(I(G)^{(s)}+I(H'))\leq {\rm reg}(I(G)^s).$$

We now consider the following short exact sequence.
\begin{align*}
0 \longrightarrow \frac{S}{(I(G)^{(s)}+I(H')):xy}(-2)\longrightarrow \frac{S}{I(G)^{(s)}+I(H')} \longrightarrow \frac{S}{I(G)^{(s)}+I(H)}
\longrightarrow 0
\end{align*}
It yields that
\begin{align*}
{\rm reg}\big(I(G)^{(s)}+I(H)\big) & \leq \max\{{\rm reg}\big(I(G)^{(s)}+I(H')\big), {\rm reg}\big((I(G)^{(s)}+I(H')):xy\big)+1\}\\ & \leq {\rm reg}(I(G)^s).
\end{align*}
\end{proof}

The following lemma has the role of the bases in the inductive argument of the proof of Proposition \ref{sum}.

\begin{lem} \label{first}
Assume that $G$ is a unicyclic graph and $C$ is the unique cycle of $G$. Let $H$ be the graph obtained from $G$ by deleting the edges of $C$. Then for every integer $s\geq 1$, we have$${\rm reg}(I(C)^{(s)}+I(H))\leq {\rm reg}(I(G)^s).$$
\end{lem}

\begin{proof}
We use induction on $|E(H)|$. If $E(H)=\emptyset$, the assertion follows from \cite[Theorem 5.9]{svv} and \cite[Theorem 5.3]{ghos}. Therefore, suppose $|E(H)|\geq 1$. We divide the proof in two cases.

\vspace{0.3 cm}
{\bf Case 1.} Assume that for every leaf $x$ of $G$, we have $d_G(x,C)=1$. Then by repeated use of Lemma \ref{del}, we conclude that$$I(C)^{(s)}+I(H)=I(G)^{(s)}+I(H).$$It then follows from Lemma \ref{first1} that$${\rm reg}(I(C)^{(s)}+I(H))\leq {\rm reg}(I(G)^s).$$

\vspace{0.3cm}
{\bf Case 2.} Assume that there exists a leaf $x$ of $G$ with $d_G(x,C)\geq 2$. Let $y$ denote the unique neighbor of $x$. In particular, $y\notin V(C)$. By \cite[Lemma 2.10]{dhs},$${\rm reg}(I(C)^{(s)}+I(H))\leq \max \big\{{\rm reg}\big((I(C)^{(s)}+I(H)):y\big)+1, {\rm reg}\big((I(C)^{(s)}+I(H)), y\big)\big\}.$$Without loss of generality, suppose $N_H(y)=\{x, x_2, \ldots, x_p\}$, for some integer $p\geq 1$. Also, set $W:=N_H(y)\cap V(C)$ (note that either $W=\emptyset$ or $|W|=1$). Then
\begin{align*}
& \big((I(C)^{(s)}+I(H)):y\big)=(I(C)^{(s)}:y)+I(H\setminus N_H[y])+(x, x_2, \ldots, x_p)\\ & =(I(C\setminus W)^{(s)}:y)+I(H\setminus N_H[y])+(x, x_2, \ldots, x_p).
\end{align*}
Therefore,
\begin{align*}
& {\rm reg}\big((I(C)^{(s)}+I(H)):y\big)={\rm reg}\big((I(C\setminus W)^{(s)}:y)+I(H\setminus N_H[y])\big)\\ & ={\rm reg}\big((I(C\setminus W)^{(s)}+I(H\setminus N_H[y])): y\big)
\end{align*}
Using \cite[Lemma 4.2]{s3} and the above equality, we conclude that$${\rm reg}\big((I(C)^{(s)}+I(H)):y\big) \leq {\rm reg}\big(I(C\setminus W)^{(s)}+I(H\setminus N_H[y])\big).$$Let $G'$ be the union of $C\setminus W$ and $H\setminus N_H[y]$. Then $G'$ is either a unicyclic graph or a forest. In the first case, using the induction hypothesis and in the second case, by \cite[Lemma 4.6 and Theorem 4.7]{bht}, we conclude that$${\rm reg}\big(I(C\setminus W)^{(s)}+I(H\setminus N_H[y])\big)\leq {\rm reg}(I(G')^s).$$Consequently,
\[
\begin{array}{rl}
{\rm reg}\big((I(C)^{(s)}+I(H)):y\big) \leq {\rm reg}(I(G')^s).
\end{array} \tag{$1$} \label{1}
\]
Since $y\notin V(C)$, we have $N_H[y]=N_G[y]$. In particular,$$G'=G\setminus N_H[y]=G\setminus N_G[y].$$It thus follows from Lemma \ref{col} and inequality (\ref{1}) that$${\rm reg}\big((I(C)^{(s)}+I(H)):y\big)+1 \leq {\rm reg}(I(G\setminus N_G[y])^s)+1\leq {\rm reg}(I(G)^s).$$

Therefore, it is enough to prove that$${\rm reg}\big((I(C)^{(s)}+I(H)), y\big)\leq {\rm reg}(I(G)^s).$$Since $y\notin V(C)$, we have
\begin{align*}
& {\rm reg}\big((I(C)^{(s)}+I(H)), y\big)={\rm reg}\big((I(C)^{(s)}+I(H\setminus y))\\ & \leq {\rm reg}(I(G\setminus y)^s)\leq {\rm reg}(I(G)^s),
\end{align*}
where the first inequality follows from the induction hypothesis, and the second inequality follows from \cite[Corollay 4.3]{bht}.
\end{proof}

\begin{lem} \label{ha}
Let $G$ be a forest. Suppose $H_1$ and $H_2$ are subgraphs of $G$ with$$E(H_1)\cap E(H_2)=\emptyset \ \ \ and \ \ \ E(H_1)\cup E(H_2)=E(G).$$Then for every integer $s\geq 1$,$${\rm reg}(I(H_1)^s+I(H_2))\leq {\rm reg}(I(G)^s).$$
\end{lem}

\begin{proof}
The assertion is essentially proved in \cite[Lemma 4.6]{bht}. The only point is that in \cite[Lemma 4.6]{bht}, it is assumed that $H_1$ and $H_2$ are induced subgraphs of $G$. However, the proof is valid for any arbitrary subgraphs.
\end{proof}

The following proposition is the main step in the proof of Theorem \ref{main} and its proof is a modification of \cite[Lemma 4.6]{bht}.

\begin{prop} \label{sum}
Let $G$ be a unicyclic graph. Suppose $H_1$ and $H_2$ are subgraphs of $G$ with$$E(H_1)\cap E(H_2)=\emptyset \ \ \ and \ \ \ E(H_1)\cup E(H_2)=E(G).$$Assume further that the unique cycle of $G$ is a subgraph of $H_1$. Then for every integer $s\geq 1$,$${\rm reg}(I(H_1)^{(s)}+I(H_2))\leq {\rm reg}(I(G)^s).$$
\end{prop}

\begin{proof}
As the isolated vertices have no effect on edge ideals, we assume that $V(H_1)=V(H_2)=V(G)$. Let $C$ denote the unique cycle of $G$. Thus, by assumption, $C$ is a subgraph of $H_1$. We use induction on $s+|E(H_1)|$. For $s=1$, there is nothing to prove. Therefore, suppose $s\geq 2$. If $E(H_1)=E(C)$, then the desired inequality follows from Lemma \ref{first}. Hence, we assume $H_1$ strictly contains $C$. As $H_1$ is a unicyclic graph, it has a leaf, say $x$. Let $y$ denote the unique neighbor of $x$ in $H_1$ and consider the following short exact sequence.
\begin{align*}
0 & \longrightarrow S/((I(H_1)^{(s)}+I(H_2)):xy)(-2)\longrightarrow S/(I(H_1)^{(s)}+I(H_2))\\ & \longrightarrow S/((I(H_1)^{(s)}+I(H_2)),xy)
\longrightarrow 0
\end{align*}
This implies that ${\rm reg}(I(H_1)^{(s)}+I(H_2))$ is bounded above by$$\max \big\{{\rm reg}((I(H_1)^{(s)}+I(H_2)):xy)+2,{\rm reg}((I(H_1)^{(s)}+I(H_2)),xy)\big\}.$$By assumption $xy$ is not an edge of $H_2$. Set$$U:=N_{H_2}[x]\cup N_{H_2}[y]$$and$$W:=(U\cap V(H_1))\setminus \{x, y\}.$$Then using Lemma \ref{colon}, we have
\begin{align*}
& ((I(H_1)^{(s)}+I(H_2)):xy)=I(H_1)^{(s-1)}+I(H_2\setminus U)+({\rm the \ ideal\ generated\ by}\ U\setminus \{x, y\})\\ & =I(H_1\setminus W)^{(s-1)}+I(H_2\setminus U)+({\rm the \ ideal\ generated\ by}\ U\setminus \{x, y\}).
\end{align*}
This yields that$${\rm reg}((I(H_1)^{(s)}+I(H_2)):xy)={\rm reg}(I(H_1\setminus W)^{(s-1)}+I(H_2\setminus U)).$$Let $G'$ be the union of $H_1\setminus W$ and $H_2\setminus U$. In fact, $G'$ is the induced subgraph of $G$ on $U\setminus \{x,y\}$. Then $G'$ is either a forest or a unicyclic graph. In the first case, by Lemma \ref{ha} and \cite[Theorem 4.7]{bht}, we conclude that$${\rm reg}(I(H_1\setminus W)^{(s-1)}+I(H_2\setminus U))\leq {\rm reg}(I(G')^{s-1})\leq {\rm reg}(I(G)^{s-1}).$$Here, the second inequality follows from \cite[Corollay 4.3]{bht} and the fact that $G'$ is an induced subgraph of $G$. Now assume that $G'$ a unicyclic graph. Then the induction hypothesis implies that
\begin{align*}
{\rm reg}(I(H_1\setminus W)^{(s-1)}+I(H_2\setminus U))\leq {\rm reg}(I(G')^{s-1})\leq {\rm reg}(I(G)^{s-1}),
\end{align*}
where, the second inequality follows from \cite[Corollay 4.3]{bht}. Therefore, in any case,$${\rm reg}((I(H_1)^{(s)}+I(H_2)):xy)={\rm reg}(I(H_1\setminus W)^{(s-1)}+I(H_2\setminus U))\leq {\rm reg}(I(G)^{s-1}),$$ and using \cite[Theorem 5.4]{abs}, we deduce that$${\rm reg}((I(H_1)^{(s)}+I(H_2)):xy)+2\leq {\rm reg}(I(G)^{s-1})+2={\rm reg}(I(G)^s).$$

Hence, it is enough to prove that$${\rm reg}((I(H_1)^{(s)}+I(H_2)),xy)\leq {\rm reg}(I(G)^s).$$ Let $H_1'$ be the graph obtained from $H_1$ by deleting the edge $xy$ and let $H_2'$ be the graph obtained from $H_2$ by adding the edge $xy$. Using Lemma \ref{del}, we have$$I(H_1)^{(s)}+I(H_2)+(xy)=I(H_1')^{(s)}+I(H_2').$$Since$$E(H_1')\cap E(H_2')=\emptyset \ \ \ and \ \ \ E(H_1')\cup E(H_2')=E(G),$$it follows from the induction hypothesis that$${\rm reg}((I(H_1)^{(s)}+I(H_2)),xy)={\rm reg}(I(H_1')^{(s)}+I(H_2'))\leq {\rm reg}(I(G)^s),$$and this completes the proof.
\end{proof}

We are now ready to prove the main result of this paper.

\begin{thm} \label{main}
Let $G$ be a unicyclic graph. Then for every integer $s\geq 1$, we have$${\rm reg}(I(G)^{(s)})={\rm reg}(I(G)^s).$$
\end{thm}

\begin{proof}
If $G$ is a cycle, the assertion is known by \cite[Theorem 5.9]{svv} and \cite[Theorem 5.3]{ghos}. Thus, assume that $G$ is not a cycle. 

The inequality ${\rm reg}(I(G)^{(s)})\leq {\rm reg}(I(G)^s)$ follows from Proposition \ref{sum} by substituting $H_1=G$ and $H_2=\emptyset$. To prove the reverse inequality, note that by \cite[Corollary 4.12]{bc} and \cite[Theorem 5.4]{abs}, either ${\rm reg}(I(G)^s)=2s+\nu(G)-1$, or ${\rm reg}(I(G)^s)=2s+\nu(G)$. We divide the remaining of the proof in two cases.

\vspace{0.3cm}
{\bf Case 1.} Assume that ${\rm reg}(I(G)^s)=2s+\nu(G)-1$. Using \cite[Theorem 4.6]{ghos}, we have$$2s+\nu(G)-1\leq {\rm reg}(I(G)^{(s)})\leq {\rm reg}(I(G)^s)=2s+\nu(G)-1.$$Consequently, ${\rm reg}(I(G)^{(s)})={\rm reg}(I(G)^s)$.

\vspace{0.3cm}
{\bf Case 2.} Assume that ${\rm reg}(I(G)^s)=2s+\nu(G)$. Let $C$ be the unique cycle of $G$ and let $\ell$ denote the length of $C$. We know from \cite[Theorem 5.4]{abs} that ${\rm reg}(I(G))=\nu(G)+2$ and it follows from \cite[Corollary 3.9]{abs} that $\ell\equiv 2$ (mod $3$) and there is a subset $\Gamma(G)\subseteq V(G)\setminus V(C)$ such that $C$ is a connected component of $G\setminus \Gamma(G)$, and $\nu(G\setminus \Gamma(G))=\nu(G)$. Since, $\ell\equiv 2$ (mod $3$) and $G\neq C$, we have $\nu(C)< \nu(G)$. In particular, $G\setminus \Gamma(G)$ strictly contains $C$. Therefore, $G\setminus \Gamma(G)$ is the disjoint union of $C$ and a forest, say $H$. Hence,$$\nu(G)=\nu(G\setminus \Gamma(G))=\nu(C)+\nu(H).$$Since $\ell\equiv 2$ (mod $3$), using \cite[Theorem 7.6.28]{j}, we have ${\rm reg}(I(C))=\nu(C)+2$. We also know from \cite[Theorems 4.7]{bht} that for every integer $k\geq 1$,$${\rm reg}(I(H)^k)=2k+\nu(H)-1.$$Hence, it follows from \cite[Corollary 4.5]{ghos} and \cite[Theorem 5.11]{hntt} that
\begin{align*}
& {\rm reg}(I(G)^{(s)})\geq {\rm reg}(I(G\setminus \Gamma(G))^{(s)})\geq {\rm reg}(I(C))+{\rm reg}(I(H)^s)-1\\ & =\nu(C)+2+2s+\nu(H)-1-1=2s+\nu(G)={\rm reg}(I(G)^s).
\end{align*}
Thus, ${\rm reg}(I(G)^{(s)})={\rm reg}(I(G)^s)$.
\end{proof}

%%%%%%%%%%%%%%%%%%%%%%%%%%%%%%%%%%%%%%%%%%%%%%%%%%%%%%%%%%%%%%%%%%%%%%%%%%

%\section*{Acknowledgment}

%%%%%%%%%%%%%%%%%%%%%%%%%%%%%%%%%%%%%%%%%%%%%%%%%%%%%%%%%%%%%%%%%%%%%%%%%%

%%%%%%%%%%%%%%%%%%%%%%%%%%%%%%%%%%%%%%%%%%%%%%%%%%%%%%%%%%%%%%%%%%%%%%%%%%

\end{document}